\documentclass[10pt]{article}

\usepackage[english]{babel}
\usepackage[ansinew]{inputenc}
\usepackage{epsfig}
\usepackage{latexsym}
\usepackage{amssymb}
\usepackage{amsmath}
\usepackage{amsthm}
\usepackage{amsfonts}
\usepackage{subfigure}
\usepackage{color}
\usepackage{lineno}

\newtheorem{theorem}{Theorem}[section]

\newtheorem{lemma}[theorem]{Lemma}
\newtheorem{prop}[theorem]{Proposition}
\theoremstyle{definition}
\newtheorem{defini}[theorem]{Definition}
\theoremstyle{remark}

\numberwithin{equation}{section}

\newcommand{\conv}{\operatorname{Conv}}
\newcommand{\bigO}[1]{\operatorname{O}\left(#1 \right)}

\title{On the Connectedness and Diameter of a Geometric Johnson Graph\thanks{
Part of the work was done in the 2nd Workshop on Discrete Geometry and its Applications. Oaxaca, Mexico, September 2009.}}

\author{
C. Bautista-Santiago \footnotemark[2]
\and J. Cano \footnotemark[2]
\and R. Fabila-Monroy \footnotemark[3]  \footnotemark[7]
\and D. Flores-Pe\~naloza \footnotemark[4]
 \and H. Gonz\'alez-Aguilar \footnotemark[2]
 \and D. Lara \footnotemark[5]
 \and E. Sarmiento \footnotemark[3]
 \and J. Urrutia \footnotemark[2] \footnotemark[6]
 }

\begin{document}

\maketitle

\def\thefootnote{\fnsymbol{footnote}}
\footnotetext[2]{Instituto de Matem\'aticas, Universidad Nacional Aut\'onoma de M\'exico, M\'exico.}
\footnotetext[3]{Departamento de Matem\'aticas, Centro de Investigaci\'on y de Estudios Avanzados del Instituto Polit\'ecnico Nacional, M\'exico.}
\footnotetext[4]{Departamento de Matem\'aticas, Facultad de Ciencias, Universidad Nacional Aut\'onoma de M\'exico, M\'exico.}
\footnotetext[5]{Universidad Aut\'onoma Metropolitana Azcapotzalco, Departamento de Sistemas.}
\footnotetext[6]{Partially supported by CONACyT of Mexico, grant CB-2007/80268.}
\footnotetext[7]{Corresponding author: \tt ruyfabila@math.cinvestav.edu.mx}

\begin{abstract}
Let $P$ be a set of $n$ points in general position 
in the plane.
A subset $I$ of $P$ is called an \emph{island}
if there exists a convex set $C$ such that $I = P \cap C$.
In this paper we define the
\emph{generalized island Johnson graph} of $P$ as the
graph whose vertex consists of all islands of $P$
of cardinality $k$, two of which are adjacent
if their intersection consists of exactly
$l$ elements. We show that  for large enough values of $n$, this graph 
is connected, and give upper and lower bounds
on its diameter.

\end{abstract}


\section{Introduction}

Let $[n]:=\{1, 2, \ldots , n\}$ and let $k \leq n$ be a positive integer. 
A $k$-subset of a set is a subset of $k$ elements.
The \emph{Johnson graph}
$J(n,k)$ is the graph whose vertex set consists
of all $k$-subsets of $[n]$, two of which are adjacent
if their intersection has size $k-1$.
The \emph{Kneser graph} $K(n,k)$ is the graph whose
vertex set consists of all $k$-subsets of $[n]$,
two of which are adjacent if they are disjoint.
The \emph{generalized Johnson graph} $GJ(n,k,l)$ is the graph 
whose vertex set consists of all $k$-subsets of $[n]$, two
of which are adjacent if they have exactly $l$ elements in common.
Thus $GJ(n,k,k-1)=J(n,k)$ and $GJ(n,k,0)=K(n,k)$.
Johnson graphs have been widely studied in the literature.
This is in part for their applications in Network Design---where
connectivity and diameter\footnote{A graph is connected if there is 
a path between any pair of its vertices. The distance
between two vertices is the length of the shortest path
joining them. The diameter is the maximum distance
between every pair of vertices of a graph.} are of importance. (Johnson
graphs have small diameter and high connectivity.)
Geometric versions of these graphs have been defined in the literature.
In~\cite{hurtado} the chromatic
numbers of some ``geometric type Kneser graphs'' were studied.
In this paper we study the connectedness and diameter
of a ``geometric'' version of the generalized Johnson graph.

Let $P$ be a set of $n$ points in the plane.
A subset $I\subseteq P$ is called an \emph{island} if there exists a convex set $C$ such that $I=P \cap C$.
We say that $I$ is a \emph{$k$-island} if it has cardinality $k$ (see Figure \ref{fig:island}).
Let $0\le l < k \leq n$ be integers. The \emph{generalized island Johnson graph} $IJ(P,k,l)$ is
the graph whose vertex set consists of all $k$-islands of $P$,
two of which are adjacent if  their intersection
has exactly  $l$ elements. Note that $IJ(P,k,l)$ is an induced
subgraph of $GJ(n,k,l)$. If $P$ is in convex position,
then $IJ(P,k,l)$ and $GJ(n,k,l)$ are isomorphic---since in this case every subset of
$k$ points is a $k$-island.

\begin{figure}
\centering
\mbox{\subfigure{\includegraphics[scale=0.8]{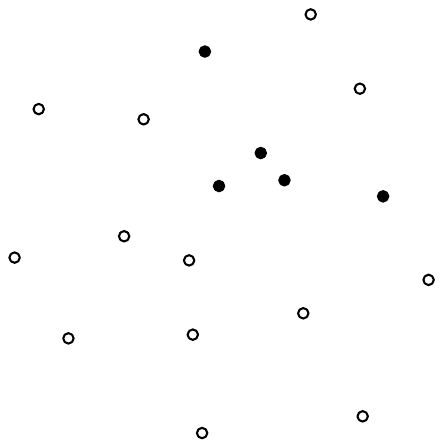}}\quad \quad \quad \quad \quad
\subfigure{\includegraphics[scale=0.8]{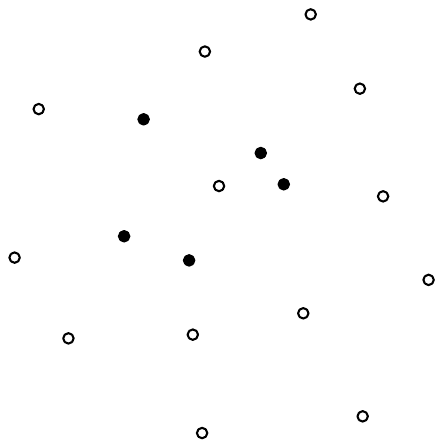}}}
\caption{A subset of $5$ points which is a $5$-island, and a subset of $5$ points which is not (both
painted black).} \label{fig:island}
\end{figure}


Graph parameters of $IJ(P,k,l)$ can be translated
to problems in Combinatorial Geometry of point sets. Here are some examples.
\begin{itemize}
\item[-]The number of vertices of this graph is the number
of $k$-islands of $P$---the problem of estimating this number
was recently studied in \cite{clemens_ruy}.

\item[-]An empty triangle of $P$ is a triangle with vertices
on $P$ and without points of $P$ in its interior.
The empty triangles of $P$ are precisely
its $3$-islands (or the number
of vertices in $IJ(P,3,l)$). Counting them has been
a widely studied problem \cite{emptysimplices,smallnumber,fewemptydumi,katmeir,minvaltr}.

\item[-] A related question~\cite{BK} is:
What is the maximum number of empty triangles
that can share an edge? This translates
to the problem of determining the clique number
of $IJ(P,3,2)$.
\end{itemize}

The paper is organized as follows.
In Section \ref{sec:conec}, we prove
that $IJ(P,k,l)$ is connected
 when $n$ is large enough with respect
 to $k$ and $l$. The proof of this result
implies an upper bound of $\bigO{\frac{n}{k-l}}+\bigO{k-l}$
on the diameter of this graph. In Section~\ref{sec:diam}, we
improve this bound for the case when $l \le k/2$, where we show that
the diameter is at most $\bigO{\log n}+\bigO{k-l}$.
We also exhibit a choice of $P$ for which $IJ(P,k,l)$ has diameter
at least $\Omega\left(\frac{\log n-\log k}{\log(k-l)}\right)$.
Note that these bounds are asymptotically tight when
$l \le k/2$ and, $k$ and $l$ are constant with respect to
$n$. A preliminary version of this paper appeared in \cite{eurocg10}.


\section{Connectedness}\label{sec:conec}

In this section we prove the following Theorem.

\begin{theorem}\label{teo:main}
 If $n > (k-l)(k-l+1)+k$, then $IJ(P,k,l)$
is connected and its
diameter is $\bigO{\frac{n}{k-l}}+\bigO{k-l}$.
\end{theorem}

The proof is divided in two parts:

\begin{itemize}
\item First we prove that $IJ(P,k,l)$ contains a connected subgraph $\mathcal{F}$
of diameter $\bigO{\frac{n-k}{k-l}}+\bigO{k-l}$.

\item Next we prove that for every vertex in $IJ(P,k,l)$ there is a path of
length at most $\bigO{\frac{n}{k-l}}$ connecting it to a  vertex in $\mathcal{F}$.

\end{itemize}

\subsection{$IJ(P,k,l)$ contains a connected subgraph}

Let $P:=\{p_0, p_1,\ldots,p_{n-1}\}$ be a set of $n$ points in general
position in the plane. So that
$p_0$ is the topmost point of $P$,
and $p_1,\ldots, p_{n-1}$
are sorted counterclockwise by angle around $p_0$. 
For $0\le i\le j \le n$, let $P_{i,j}:=\{p_i, p_{i+1}, \ldots, p_j \}$ and let $P'_{i,j}:= P_{i,j} \cup \{p_0\}$.
Observe that $P_{i,j}$ and $P'_{i,j}$ are both islands of $P$. 
We call these two types of islands \emph{projectable}.
Projectable islands are those islands that can be ``projected'' 
to a horizontal line, in such a way that its points are consecutive in the image of the whole set;
see Figure \ref{fig:projectable}.
Let $\mathcal{F}$ be the subgraph of $IJ(P,k,l)$ induced by the projectable
$k$-islands of $P$. 
Let $S$ be a set of $n-1$ points on a horizontal line $h$, and let $S':=S \cup \{x\}$, where $x$ is a point not in $h$. 
It is not hard to see that $\mathcal{F}$ is isomorphic to $IJ(S',k,l)$. We  classify the islands of $S'$ into two types: those that contain $x$, and those that do not. Notice that these two types correspond to the two types of  projectable islands of $P$. 

\begin{figure}
\centering
\mbox{\subfigure[Projectable.]{\includegraphics[scale=1.0]{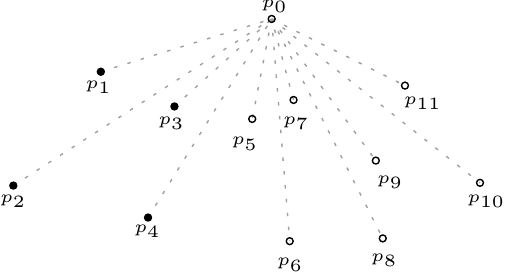}} \quad \quad \quad
\subfigure[Non projectable.]{\includegraphics[scale=1.0]{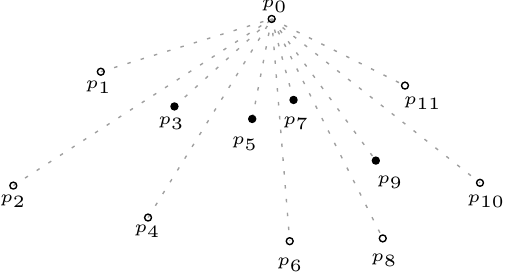}}}
\caption{A $4$-island which is projectable, and a $4$-island which is not projectable.} \label{fig:projectable}
\end{figure}

Now we show that $IJ(S',k,l)$ is connected.
 First we consider the subgraph of $IJ(S',k,l)$ induced
by those islands of $S'$ that do
not contain $x$. Note that this is precisely $IJ(S,k,l)$.
Without loss of generality assume that $S$ is a set  $x_1 < x_2 < \cdots < x_{n-1}$ of points on the real line.
Observe that a $k$-island of $S$ is an \emph{interval}
of $k$ consecutive elements $\{x_i, \ldots , x_{i+k-1}\}$. 
For the sake of clarity, in what follows we refer to $k$-islands of $S$ as $k$-intervals.

Two $k$-intervals of $S$ are adjacent
in $IJ(S,k,l)$ if they overlap in exactly $l$ elements.
It follows easily that if $l>0$, each $k$-interval 
is adjacent to at most two different $k$-intervals, one containing its first element,
and the other containing its last; see Figure~\ref{figone}.
Since $IJ(S,k,l)$ has no cycles and its maximum
degree is at most two, it is a union of pairwise disjoint paths. These paths can be described as follows.
For $i<j$, let $A_i$ and $A_j$ be the intervals ending at 
$x_i$ and $x_j$ respectively. There is a path between $A_i$ and $A_j$ if and only if $i \equiv j \mod(k-l)$. 
For consider the interval adjacent with $A_i$ to its right, 
this interval must end at point $x_{i+(k-l)}$ 
(leaving exactly $l$ points on the intersection). 
On the other hand, the interval adjacent with $A_i$ to its left, 
must end at point $x_{i-(k-l)}$; see Figure~\ref{figone}.
For  $0 \leq r < k-l$, let  $\mathcal{P}_r$ be the 
subgraph of $IJ(S,k,l)$ induced by those $k$-intervals
ending at a point with index congruent to $r \mod (k-l)$.
Thus we have:

\begin{figure}
\begin{center}
  \includegraphics[scale=0.8]{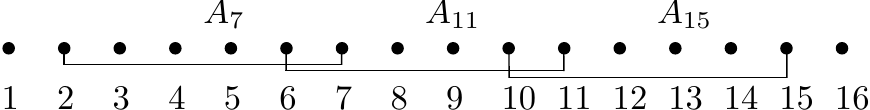}
  \caption{Three different $k$-intervals, for $|S|=16, k=6, l=2$.}\label{figone}
\end{center}
\end{figure}

\begin{figure}
\begin{center}
 \includegraphics[scale=0.8]{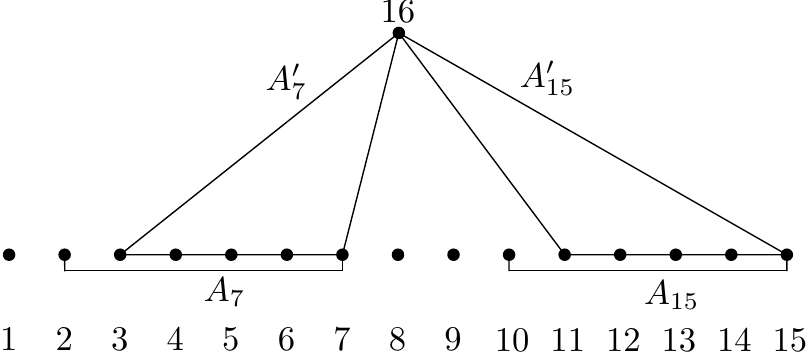}
  \caption{Four different $k$-islands in $S'$, for $|S'|=16, k=6, l=2$.}\label{fig:both_types_of_islands}
\end{center}
\end{figure}

%

\begin{prop}\label{prop:collinear-points}
If $l>0$, $\mathcal{P}_r$ is an induced path of $IJ(S,k,l)$.
Moreover $IJ(S,k,l)$ is the union of $\{\mathcal{P}_r\,|\,0\le r < k-l\}$. \qed
\end{prop}



For $l=0$ and $n \geq 3k-1$, every $k$-interval would 
either intersect the left-most $k$-interval or
the right-most $k$-interval, but not both. In this
case   $IJ(S,k,0)$ is connected and its diameter is at most $3$.
Note that except for some special cases, $IJ(S,k,l)$
is disconnected. Remarkably, as we show next, for a large enough value of $n$, 
the addition of one extra point makes the graph connected. 

As before  $A_i$ is the $k$-island ($k$-interval) that ends at point $x_i$ and does not contain $x$.
Let $A'_i$ be the $k$-island ending at point $x_i$ and containing $x$; 
see Figure~\ref{fig:both_types_of_islands}. The structure of $IJ(S',k,l)$ when $l<2$ is different
from when $l\ge 2$.  In what follows, we
assume that $l \ge 2$ and  briefly discuss the case $l <2$ at the
end of this section. Note that the subgraph of $IJ(S',k,l)$ induced by the $A_i$'s is
precisely $IJ(S,k,l)$. While the subgraph induced
by the islands $A'_i$ is isomorphic to $IJ(S,k-1,l-1)$. 
From these observations, the following lemma is not hard to prove:

\begin{lemma}\label{lem:grid}
If $l\geq 2$, then  in $IJ(S',k,l)$:
\begin{enumerate}
\item $A_i$ is adjacent to $A'_{i-(k-l)}$  and $A'_{i+(k-l)-1}$ (if they exist).
\item $A'_i$ is adjacent to $A_{i+(k-l)}$ and $A_{i-(k-l)+1}$ (if they exist).
\end{enumerate}
\end{lemma}



The following theorem provides sufficient and necessary conditions
for \\$IJ(S',k,l)$ to be connected.

\begin{theorem}\label{G_kl_S'_connected}
For $l \geq 2$, the graph $IJ(S',k,l)$ is connected if and only if
$n\geq 3k-2l-1$ or $n=k$.
\end{theorem}
\begin{proof}
Let $I$ and $J$
be two $k$-islands of $S'$. As long as the intermediate
islands exists we can repeatedly use Lemma~\ref{lem:grid} to find a path 
from $I$ to an island whose endpoint is in the same 
residue class of $(k-l)$ as the endpoint of $J$, and that contains
$x$ if and only if $J$ does. This is the case whenever $n \ge 3k-2l-1$.
 Afterwards Proposition~\ref{prop:collinear-points} ensures that there is a path
from this island to $J$.

Suppose that $n < 3k-2l-1$, then there exists a
$k$-island containing $x$ and having
less than $k-l$ points to its left and less than
$k-l$ points to its right. This island is an isolated vertex
in $IJ(S',k,l)$. This sole vertex is all of $IJ(S',k,l)$ when
$n=k$ (in which case the graph is connected). However, if $n > k$, there
are at least two such $k$-islands.
\end{proof}

The proof of Theorem~\ref{G_kl_S'_connected} implicitly provides
the following bound on the diameter of $IJ(S',k,l)$.

\begin{prop}\label{prop:diamL}
If $IJ(S',k,l)$ is connected, then its diameter is
$\bigO{\frac{n-k}{k-l}}+\bigO{k-l}$. 
\end{prop}
\begin{proof}
Suppose that $n \ge 3k-2l-1$, as otherwise the bound trivially holds.
Let $I$ and $J$ be two $k$-islands of $S'$. Note that
it takes at most $2(k-l)$ applications
of Lemma~\ref{lem:grid} to take $I$ to an island whose endpoint is in the same 
residue class of $(k-l)$ as the endpoint of $J$, and that contains
$x$ if and only if $J$ does. The path in  $IJ(S',k,l)$ or in $IJ(S,k,l)$---depending 
on whether $J$ contains $x$ or not--connecting this island
to $J$ has length at most $\left\lceil \frac{n-k}{k-l}\right\rceil$. Since
between the starting points of two consecutive intervals in
the same residue class there are $k-l$ points, none of these
subintervals, lies in the right-most.

\end{proof}

Finally we consider the case when $l<2$. As we mentioned before, if $l=0$ and $n\ge 3k-1$, 
$IJ(S,k,0)$ is connected and its diameter is at most $3$. This
 is the case also for   $IJ(S',k,0)$.
On the other hand, if $l=1$, then the islands containing
$x$ induce a graph isomorphic to $IJ(S,k-1,0)$.
From these observations and Lemma~\ref{lem:grid}, we get the following result. 

\begin{prop}
If $n \ge 3k$, $IJ(S',k,0)$  and $IJ(S',k,1)$ are connected
and of diameter at most $4$. \qed
\end{prop}

\subsection{Paths between projectable and non projectable islands}\label{sec:prove_shrk}

To finish the proof of Theorem~\ref{teo:main}, we prove that for any island of $P$, there is
a path connecting it to a projectable island. 
At the end of this section we present a first bound on the diameter of $IJ(P,k,l)$.

Recall that $P=\{p_0,p_1,\dots,p_{n-1}\}$; $p_0$ is its topmost
point and $p_1,\dots,p_{n-1}$ are sorted counterclockwise
by angle around $p_0$. 
Let $A$ be an island of $P$
such that $|A \setminus \{p_0\}| \geq 2$. 
Define the \emph{weight} of $A$
as the difference between the largest and the
smallest indices of the
elements of $A\setminus \{p_0\}$---an island of weight
$k-1$ is always projectable. The following lemma 
ensures the existence of a path between any island and a projectable island,
by eventually reducing the weight of any given island.

\begin{lemma}[Shrinking Lemma]\label{lem:shrk}
If $n > (k-l)(k-l+1)+k$, then  every non projectable
$k$-island $A$ of $P$ has a neighbor in $IJ(P,k,l)$ which is either
a projectable island or an island whose weight
is less than that of $A$ by at least $k-l$.
\end{lemma}
\begin{proof}

Let the elements of $A$ 
different from $p_0$ be $p_{i_1}, \dots ,p_{i_m}$.
Thus $m$ is equal to $k$ or to $k-1$ depending
on whether $A$ contains $p_0$ or not.
Consider all maximal intervals of $P \setminus \{p_0\}$
containing exactly $l$ elements of $A$.
(That is maximal sets of consecutive elements of $P\setminus \{p_0\}$ containing
exactly $l$ elements of $A$.)

We distinguish two of these intervals: the one containing the first point
of $P \setminus \{p_0\}$ and the one containing the last. 
We refer to them as \emph{end intervals}, and to the rest as \emph{interior intervals}.

Note that there are at most $k-l+1$  such intervals and that
every element of $P \setminus \{p_0\}$
is in at least one of them. Since $n > (k-l)(k-l+1)+k$, 
one of these intervals, $I$, must contain
at least $(k-l)$ points of $P\setminus A$.

Suppose that $I$ is an interior interval. Let $J:=A\cap I$. If $J$ is non
empty let
$B$  be the set of the
$k-l$ points of $I\setminus A$
closest\footnote{The distance between $Conv(J)$  
and a point $p \notin Conv(J)$, is defined as the 
length of the shortest line segment having $p$ 
and a point of $Conv(J)$ as endpoints.} to  $\conv(J)$. 
If $J$ is empty then let $B$ be any $k$-island contained
in $I\setminus A$.  Then $J\cup B$ is a $k$-island adjacent to $A$ in $IJ(P,k,l)$, and its weight
is smaller than the weight of $A$ by at least
$k-l$.

Now suppose that $I$ is an end interval, let $p_S$ and $p_E$
be the first and last elements in $A \cap I$.
If $[p_S,p_E]$ contains at least $k-l$
elements of $P \setminus A$, then  proceed
as with interior intervals.
Otherwise, there are $r < k-l$ points
of $P\setminus A$ in $I$. If
$I$ is the first interval, then  let  $B$ 
be the $k-l-r$ points previous  to
$p_S$ in $P\setminus\{p_0\}$. If $I$ is the last interval, then
let $B$ be the $k-l-r$ points after
$p_E$. Note that in either case, $[p_S,p_E] \cup B$
is a projectable island adjacent
to $A$.
\end{proof}

We are ready to finish the proof of Theorem \ref{teo:main}.\\ \newline
{\bf Theorem  \ref{teo:main}} \emph{If $n > (k-l)(k-l+1)+k$, then $IJ(P,k,l)$
is connected and its
diameter is $\bigO{\frac{n}{k-l}}+\bigO{k-l}$.}
\begin{proof}
Let $A$ and $B$ be $k$-islands of $P$.  We apply
Lemma \ref{lem:shrk} successively
to find a sequence of consecutive
adjacent islands
$A=A_0,A_1,\dots,A_m$ and
$B=B_0,B_1,\dots,B_{m'}$,
in which each element has weight
smaller than the previous by at least
$k-l$, and the last element is a
projectable island. Since the weight of the initial
terms is at most $n$,
these sequences have length $\bigO{\frac{n}{k-l}}$.

As noted before the subgraph induced by the
projectable islands is isomorphic to
$IJ(S',k,l)$. Simple
arithmetic shows that if $n > (k-l)(k-l+1)+k$, then
$n > 3k-2l-2$. Thus this subgraph
is connected and has diameter
$\bigO{\frac {n-k}{k-l}}+\bigO{k-l}$ (Theorem \ref{G_kl_S'_connected}
and Proposition \ref{prop:diamL}). Hence the diameter of $IJ(P,k,l)$ is $\bigO{\frac{n}{k-l}}+\bigO{k-l}$ as claimed.
\end{proof}


\section{Bounds}\label{sec:diam}

\subsection{Upper bound}

In this section, for the case when $l \le k/2$, we improve the
upper bound on the diameter of $IJ(P,k,l)$ given in Theorem \ref{teo:main}.
We use a \emph{divide and conquer} strategy. 
Let $A$ and $B$ be two vertices of $IJ(P,k,l)$.
First we find a neighbor of $A$ and a neighbor of $B$;
discarding half of the points of $P$ in the process.
We iterate on the new found neighbors. Just before $P$ has 
very few points and $IJ(P,k,l)$ may be disconnected;
we apply Theorem \ref{teo:main}.

The following lemma provides the divide
and conquer part of the argument.
The proof uses some of the ideas
of the proof of the Shrinking Lemma.

\begin{lemma}\label{lem:divandconq}
Let $A$ and $B$ be two vertices of $IJ(P,k,l)$.
If $n \ge 2((k-l)(k-l+1)+k)$, and $l \le k/2$, then
there exists a closed halfplane, $H$, containing
at most $n/2$ and at least $(k-l)(k-l+1)+k$ points of $P$. With the additional
property that $A$ and $B$, each have a neighbor contained entirely in $H$.
\end{lemma}

\begin{proof}
We use the ham-sandwich theorem to
find a line $\ell$ so that each of the two closed halfplanes
bounded by $\ell$ contain $\lceil k/2 \rceil$
points of $A$  and  $\lceil k/2 \rceil$
 points of $B$.

Without loss of generality suppose that the halfplane $H$ above $\ell$ contains
at most $n/2$ points of $P$. If $H$, however,
does not contain at least $(k-l)(k-l+1)+k$ points of $P$,
move $\ell$ parallel down until it does. In this
case $H$ would contain at least as many points
of $A$ and $B$ as it previously did and
since we are assuming that $n \ge 2((k-l)(k-l+1)+k)$, it
still contains at most $n/2$ points of $P$.

We will now show the existence of a neighbor
of $A$ in $IJ(P,k,l)$ with the desired properties.
The corresponding neighbor of $B$ can be found
in a similar way. Let $P':=P \cap H$ and sort its elements
by distance to $\ell$. As
in the proof of Lemma \ref{lem:shrk}
we consider maximal intervals of $P'$
containing exactly $l$ consecutive elements of $A$.
There is at least one such interval, given that
$H$ contains at least $k/2$ points of $A$ and
that we are assuming $l\le k/2$. The rest
of the proof employs the same arguments
as the proof of Lemma \ref{lem:shrk} to
find a neighbor of $A$ contained in one of
these intervals.

\end{proof}

\begin{theorem} \label{thm:upb}
If $n \ge 2((k-l)(k-l+1)+k)$ and $l \leq k/2$, then the diameter
of $IJ(P,k,l)$ is $\bigO{\log n}+\bigO{k-l}$.
\end{theorem}

\begin{proof}
Consider the following algorithm. Let $A$ and $B$ be two $k$-islands
of $P$. We start by setting
$A_0:=A$, $B_0:=B$, $P_0:=P$, $n_0:=n$.
While $n_{i} \ge 2((k-l)(k-l+1)+k)$, we apply Lemma
\ref{lem:divandconq} to $P_i$,
$A_i$, and $B_i$.
At each step we obtain a closed halfplane $H_i$
containing  at most $n_i/2$
and at least $(k-l)(k-l+1)+k$ points of $P_i$,
with the additional property that
both $A_i$ and $B_i$ have neighbors
$A_{i+1}$ and $B_{i+1}$ 
contained entirely in $H_i$.
We set $P_{i+1}:=H_i \cap P_i$,
$n_{i+1}:=|P_{i+1}|$,
and continue the iteration. We can  do this procedure at most $\bigO{\log n}$
times.
In the last iteration, we have
a point set $P_m$ with fewer than
$2((k-l)(k-l+1)+k)$ and
at least $(k-l)(k-l+1)+k$ elements.
The islands $A_m$ and $B_m$ are both
contained in $P_m$, and are joined
by paths of length $\bigO{\log n}$
to $A$ and $B$ respectively.
We apply Theorem
\ref{teo:main} to obtain a path
of length at most $\bigO{k-l}$ from
$A_m$ to $B_m$. Concatenating the three paths
we obtain a path of length
$\bigO{\log n}+\bigO{k-l}$
from $A$ to $B$ in $IJ(P,k,l)$.
\end{proof}

\subsection{Lower bound}\label{sec:lower_bound}

For the lower bound we
use Horton sets~\cite{horton}.
We base our exposition on~\cite{lectures}.
Let $X$ and $Y$ be two point sets 
in the plane. We say
that $X$ is \emph{high above} $Y$ (and that $Y$ is
\emph{deep below} $X$), if the following conditions are met:

\begin{itemize}
\item No line passing through a pair of points of $X\cup Y$ is
vertical.
\item Each line passing through a pair of points of $X$ lies above
all the points of $Y$.

\item Each line passing through a pair of points of $Y$ lies below
all the points of $X$.
\end{itemize}

For a set $X=\{x_1,x_2,\dots,x_n\}$ of points 
in the plane with no two points having the same
$x$-coordinate and with the indices chosen
so that the $x$-coordinate of $x_i$
increases with $i$, we define the sets
$X_0=\{x_2,x_4,\dots\}$ (consisting of the
points with even indices) and 
$X_1=\{x_1,x_3,\dots\}$ (consisting of the
points with odd indices). Thus $X_{00}=\{x_4,x_8,\dots\}$,
$X_{01}=\{x_2,x_6,\dots\}$, $X_{10}=\{x_3,x_7,\dots\}$ and
$X_{11}=\{x_1,x_5,\dots\}$.

\begin{defini}
A finite set of points $H_0$, with not two 
of points having the same $x$-coordinate, is said to be a Horton
set if $|H_0|\le 1$, or the the following conditions
are met:
\begin{itemize}
\item Both $H_{00}$ and $H_{01}$ are Horton sets.
\item $H_{00}$ is high above $H_{01}$.
\end{itemize}
\end{defini}

Horton sets of any size were shown to exist
in \cite{horton}. We remark that in~\cite{lectures}, in the definition of
Horton sets, the second condition is that
``$H_{00}$ is high above $H_{01}$ or $H_{01}$ is high above $H_{00}$''.
For our purposes we need to fix one of these two options.

Let $H_0:=\{x_1, \ldots , x_n\}$ be a Horton set of $n$ points.
Given an island $A$ of $H_0$, we define
its \emph{depth}, $\delta(A)$, to be
the length of the longest string $s:=00\dots0$ of all zeros such that $H_s$
contains $A$. Thus for example, $\delta(x_1)=1, \delta(x_2)=2, \delta(x_3)=1, \delta(x_4)=3, \cdots$~;
refer to Figure~\ref{depth}. Note that the depth of an
island is the depth of its shallowest
point. 

\begin{center}
\begin{figure}[h]
  \includegraphics[scale=0.6]{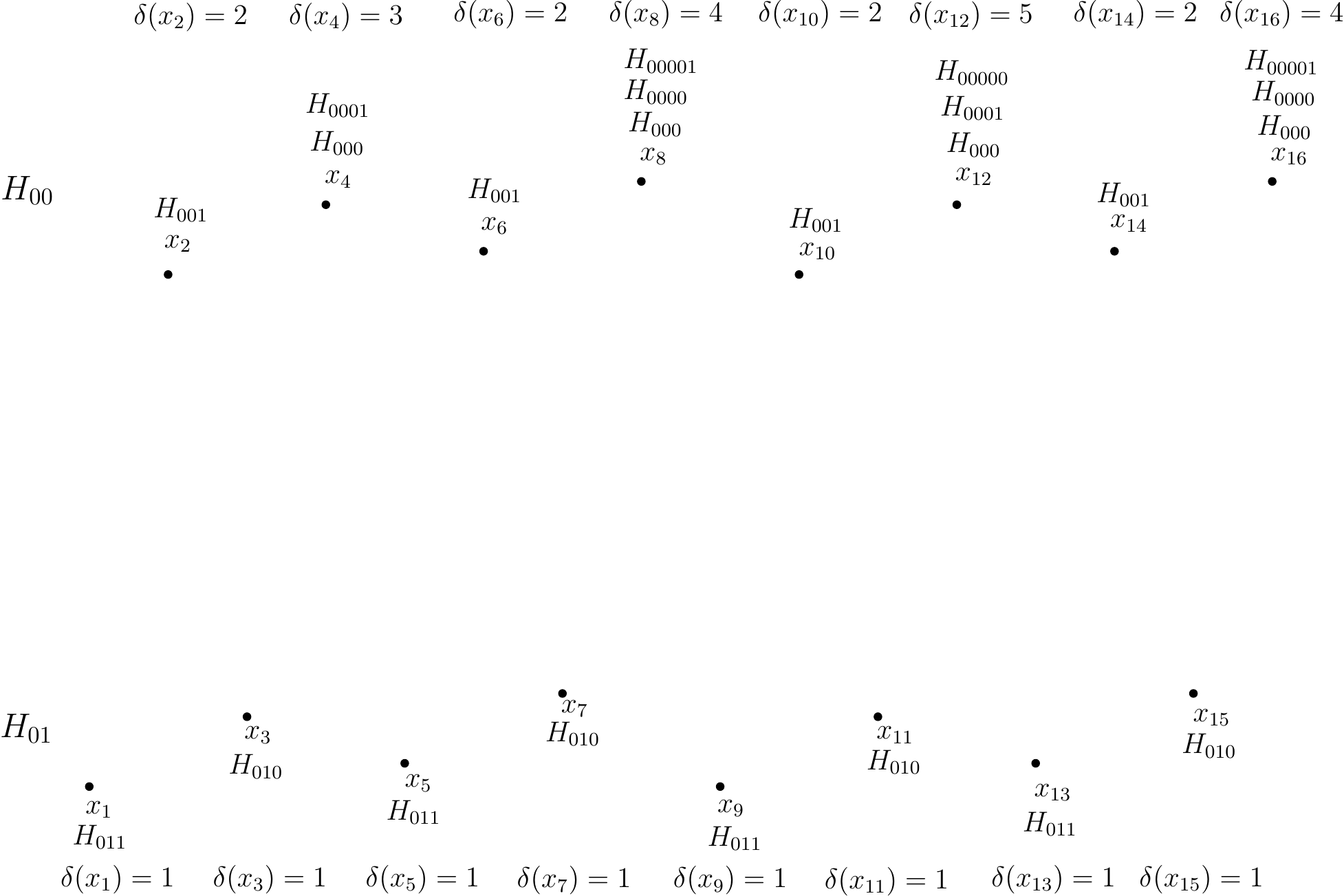}
  \caption{A Horton set with $16$ points, and the depth of its elements.}\label{depth}
\end{figure}
\end{center}

\begin{lemma} \label{lem:depth}
Let $x$ and  $y$ be two points of $H_0$, such
that $x$ is to the left of $y$,
and $z$ is  a point with depth
less than $\delta(\{x,y\})$. Then
the island $H_0 \cap \conv(\{x,y,z\})$
contains at least $2^{\delta(\{x,y\})-\delta(z)-1}-1$
points with depth greater than that of  $z$,
and lying in between $x$ and $y$.
\end{lemma}

\begin{proof}

Let $A:=H_0\cap \conv(\{x,y,z\})$.
We will proceed by induction on
$r=\delta(\{x,y\})-\delta(z)$.
If $r=1$ there is nothing to prove,
since $2^{r-1}-1=0$. Assume then that $r>1$. Let $s$ be the 
unique string of all zeros, such that
$H_s$ contains $A$ but $H_{s0}$ and $H_{s1}$ do not. Note that $z$ lies
in $H_{s1}$ while $x$ and $y$ both lie in $H_{s0}$;
actually since we are assuming $r>1$, they both
lie in $H_{s00}$. Consider the set
 $H_{s0}$, since it 
is a Horton set, $A$ contains at least a point in
$H_{s01}$, between $x$ and $y$. Of all such points,
choose $x_{k' }$ to be the shallowest.
The depth of $x_{k'}$ is one more than that 
of $z$. By induction, the island
$H_0 \cap \conv(\{x,y,x_{k'}\})$ contains
a set $I$ of at least $2^{r-2}-1$ points. These points have
depth greater than that of $x_{k'}$ (thus contained in $H_{s00}$) and between $x$ and $y$.
Therefore $I$ is contained in $A$.
For each point in $I$, consider the next point $x_m$ to its right
in $H_{s0}$. This point must be in $H_{s01}$ and $\delta(z) < \delta(x_m)$.
Thus we have $2^{r-2}-1$ additional points in $A$.
Note that the point to the right of $x$ in $H_{s0}$
is not in the previous counting.
Therefore $A$ has at least
$2^{r-2}-1+2^{r-2}-1+1=2^{r-1}-1$ points with
depth greater than that of
$z$, and lying between $x$ and $y$.
\end{proof}

\begin{lemma}\label{lem:n_depth}
If $A$ and $B$ are two adjacent islands
in $IJ(H_0,k,l)$ (with $l \ge 2$), then
their depths differ by at most $\bigO{\log(k-l)}$.
\end{lemma}
\begin{proof}
Without loss of generality assume that
the depth of $A$ is greater than
the depth of $B$.
Let $C$ be the island $A \cap B$.
Note that the depth of $C$ is at least
the depth of $A$.
If $z$ is the shallowest point of
$B$, then $\delta(z)=\delta(B)$, and this point has depth less than 
$\delta(C)$. Consider an edge of the convex hull of $C$,
whose supporting line separates $C$ and $z$. Let
$x$ and $y$ be its endpoints.
Then by Lemma \ref{lem:depth}, the island $H_{0} \cap \conv(\{x,y,z\})$
contains at least $2^{\delta(\{x,y\})-\delta(z)-1}-1\ge 2^{\delta(A)-\delta(B)-1}-1$
points, none of which is in $C$. However, since these points 
do lie in $B$, there are at most $k-l$ of them. Therefore
$\delta(A)-\delta(B)$ is $\bigO{\log(k-l)}$ as claimed.
\end{proof}

\begin{theorem} \label{thm:lwb}
The diameter of $IJ(H_{0},k,l)$ for $l\ge 2$  is
$\Omega\left(\frac{\log(n)-\log(k)}{\log(k-l)}\right)$.

\end{theorem}
\begin{proof}

Let $A$ be an island with the largest
possible depth, which is $\lceil \log_2(n/k) \rceil$.
Let $B$ be an island of depth $1$. By Lemma~\ref{lem:n_depth}
in any path joining $A$ and $B$ in
$IJ(H_{0},k,l)$, the depth of two consecutive vertices
differs by $\bigO{\log(k-l)}$. Therefore
any such path must have length $\Omega\left(\frac{\log(n)-\log(k)}{\log(k-l)}\right)$.
\end{proof}

We point out that there was an error in the proof of Theorem 11
in the  preliminary
version of this paper \cite{eurocg10}; thus the bounds
stated there are incorrect.

The diameter of the generalized Johnson graph
can be substantially different from that of the
generalized island Johnson graph.
The diameter of $GJ(n,k,l)$ is
$\bigO{k}$ when $n$ is large enough,
while the diameter of $IJ(P,k,l)$
can be $\Omega\left(\frac{\log(n)-\log(k)}{\log(k-l)}\right)$.

Determining upper and lower bounds for the
diameter of $IJ(P,k,l)$ seems to be a challenging problem when
$l > k/2$.
It might happen that there is a sharp
jump in the diameter when $l$ rises
above $k/2$. We leave the closing of this gap as an open problem.

\bibliographystyle{plain}
\bibliography{biblio}

\end{document}